\documentclass[11pt,letterpaper,reqno]{amsart}

\usepackage{amsmath,amssymb,amsfonts,amsthm}
\usepackage{bbm}
\usepackage{enumitem}
\usepackage{booktabs}
\usepackage{graphicx}
\usepackage[T1]{fontenc}
\usepackage{float}

\usepackage{tikz}
\usetikzlibrary{positioning,shapes.geometric,arrows.meta,calc,angles,quotes}
\usepackage{pgfplots}
\pgfplotsset{compat=1.16}

\usepackage{hyperref}
\usepackage{doi}
\hypersetup{
	pdfstartview={FitH},
	colorlinks=true,
	linkcolor=blue,
	citecolor=blue,
	urlcolor=blue
}

\newcommand{\C}{\mathbb{C}}

\newcommand{\T}{\mathbb{T}}
\newcommand{\Dopen}{\mathbb{D}}
\newcommand{\udisc}{\overline{\Dopen}}
\newcommand{\conv}{\operatorname{conv}}
\newcommand{\abs}[1]{\lvert #1 \rvert}

\newcommand{\Dshape}{\mathcal{D}}

\newtheorem{thm}{Theorem}[section]
\newtheorem{lem}[thm]{Lemma}

\newtheorem{cor}[thm]{Corollary}

\theoremstyle{definition}
\newtheorem{eg}[thm]{Example}
\newtheorem{rem}[thm]{Remark}
\newtheorem{defin}[thm]{Definition}

\numberwithin{equation}{section}

\newcommand{\Lu}{\mathcal{L}(u)}

\begin{document}
	\title[Angle duality for incomplete polynomials]{Angle duality and a gap principle for convex combinations of incomplete polynomials on the unit circle}

\author[T.~Zhang]{Teng Zhang}
\address{School of Mathematics and Statistics, Xi'an Jiaotong University, Xi'an 710049, P.~R.\ China}
\email{teng.zhang@stu.xjtu.edu.cn}


 \begin{abstract}
		In this paper, we establish an angle duality and a gap principle for convex combinations of incomplete polynomials, extending two results of Ge and Gonek in [IMRN, 2024].
		Our approach is geometric: we introduce an ``angle gain'' mechanism for points inside the lune region and quantify how moving away from the unit circle forces a definite increase in the relevant angle functional.
		This yields a robust lower bound that is uniform under convex mixing and leads to the desired separation  phenomenon.
		The main difficulty is that incomplete polynomials and their convex combinations may have highly nonuniform root distributions on the unit circle, so classical convex-hull type constraints  are too coarse; one must instead control the local geometry of chords and boundary arcs and relate it to critical-point behavior through sharp trigonometric inequalities.
	\end{abstract}

\maketitle
		\section{Introduction}
		
	The geometry of zeros and critical points of complex polynomials is a classical topic in complex analysis.
One of the most fundamental results is the Gauss--Lucas theorem \cite[p.~18]{BE95}, which states that the critical points of any nonconstant polynomial $p(z)$ lie in the convex hull of its zeros.
	In particular, if all zeros of $p$ lie on the unit circle, then all critical points necessarily lie in the closed unit disk.
The	Gauss--Lucas theorem has been generalized in many directions; see, for instance, \cite{DE08,Sen21,Zha24,Zha26}.  Much of the modern progress on locating critical points of complex polynomials is due to  Sendov~\cite{Hay19}, Schmeisser~\cite{Sch77}, Schoenberg~\cite{Sch86}, and others~\cite{BX99,BS99,CN06,KPPSS11,KT16,Mal05,Paw98,Per03,Tang25,TZ25,Tao22,Zha26,Zha26+,Zha26++}, who pushed the subject beyond global convex-hull restrictions toward genuinely quantitative, local control.

	In this paper, we focus on polynomials whose zeros lie on the unit circle.
	Such polynomials occur in several areas of mathematics: for instance, after a suitable normalization, Dirichlet $L$-functions over function fields have all zeros on the unit circle; characteristic polynomials of unitary matrices have all zeros on the unit circle; and partition functions in statistical mechanics can have all zeros on the unit circle via the Lee--Yang theorem \cite{LY52}.
	
	Throughout, let $\mathcal{L}(u)$ be a monic polynomial of degree $N:=\deg\mathcal{L}\ge 2$ whose zeros all lie on the unit circle $\abs{u}=1$.
	By the Gauss--Lucas theorem, all critical points of $\mathcal{L}$ lie in the closed unit disk
	\[
	\udisc:=\{u\in\C:\ \abs{u}\le 1\},
	\]
	and we write
	\[
	\T:=\{u\in\C:\ \abs{u}=1\}
	\]
	for the unit circle.
	
	Denote the zeros of $\Lu$ by $z_1,z_2,\dots,z_N$, \emph{counted with multiplicity}, and the critical points (zeros of $\mathcal{L}'$) by $z_1',z_2',\dots,z_{N-1}'$, counted with multiplicity, so that
	\[
	\mathcal{L}(u)=\prod_{j=1}^{N}(u-z_j),
	\qquad
	\mathcal{L}'(u)=N\prod_{k=1}^{N-1}(u-z_k').
	\]
	Write
	\[
	z_k'=\abs{z_k'}\, e^{i\theta_k'}.
	\]
	(If some $z_k'=0$, we may assign $\theta_k'=0$ (or any value), since only $\abs{z_k'}$ and the angles $\Theta(z_k';z,z^{+})$ will be used.)
	
	Let $\{\zeta_1,\dots,\zeta_M\}$ be the set of \emph{distinct} zeros of $\mathcal{L}$ (so $1\le M\le N$).
	Write
	\[
	\zeta_j=e^{i\theta_j},\qquad 0\le \theta_1< \theta_2< \cdots < \theta_M<2\pi,
	\qquad
	\theta_{M+1}:=\theta_1+2\pi,
	\]
	and adopt cyclic indexing for the $\zeta_j$.
	In this paper, 
	we will often focus on a pair of \emph{consecutive distinct} zeros $\zeta_j$ and $\zeta_{j+1}$ of $\mathcal{L}$ (with cyclic indexing).
	In that case, we write
	\[
	z=e^{i\theta}\quad\text{for }\zeta_j,
	\qquad
	z^{+}=e^{i\theta^{+}}\quad\text{for }\zeta_{j+1},
	\]
	where $\theta<\theta^{+}<\theta+2\pi$.
	
	For three distinct points $a,b,c\in\C$, let $\Theta(a;b,c)\in[0,\pi]$ denote the (unoriented) angle between the vectors $\overrightarrow{ab}$ and $\overrightarrow{ac}$.
	
	Throughout, $\arg$ denotes the principal argument taking values in $(-\pi,\pi]$.
	When comparing arguments we work modulo $2\pi$, writing congruences $\equiv \pmod{2\pi}$.
	
	\begin{defin}[Endpoint convention]\label{def:endpoint}
		Let $z=e^{i\theta}$ and $z^{+}=e^{i\theta^{+}}$ be distinct points on $\T$ with $\theta<\theta^{+}<\theta+2\pi$.
		If $a\in\{z,z^{+}\}$, we set
		\[
		\Theta(a;z,z^{+}):=\frac{\theta^{+}-\theta}{2}.
		\]
		This convention is chosen so that our angle identities extend naturally to endpoint cases (e.g., when a zero of $\mathcal{L}_\lambda$ or a critical point of $\mathcal{L}$ coincides with $z$ or $z^{+}$).
	\end{defin}
	
	Ge and Gonek \cite{GG24} established the following \emph{angle duality theorem} relating the zeros and critical points of $\Lu$. 
	
	\begin{thm}[Ge--Gonek]\label{thm:Ge_Gonek}
		Let $z=e^{i\theta}$ and $z^{+}=e^{i\theta^{+}}$ be two consecutive \emph{distinct} zeros of $\mathcal{L}(u)$.
		Then
		\begin{equation*}
			\sum_{k=1}^{N-1}\Theta(z_k';z,z^{+})
			=\pi+(N-2)\frac{\theta^{+}-\theta}{2},
		\end{equation*}
		where the critical points $z_k'$ are counted with multiplicity and angles at endpoints are interpreted using Definition~\ref{def:endpoint} when needed.
	\end{thm}
	Figure~\ref{fig:three-angles-pi} illustrates the case $N=3$ in Theorem~\ref{thm:Ge_Gonek}.
	\begin{figure}[htbp]
		\centering
		\begin{tikzpicture}[scale=2.8]
			\coordinate (O) at (0,0);
			
			\draw[black,dashed,thick] (O) circle (1);
			
			\coordinate (z)  at (1,0);
			\coordinate (zp) at (0,1);
			\coordinate (zm) at (0,-1);
			
			\draw[red,fill=white,thick] (z)  circle (0.012);
			\draw[red,fill=white,thick] (zp) circle (0.012);
			\draw[red,fill=white,thick] (zm) circle (0.012);
			
			\node[red,font=\small,anchor=west]  at (z)  {$z=1$};
			\node[red,font=\small,anchor=south] at (zp) {$z^{+}=i$};
			\node[red,font=\small,anchor=north] at (zm) {$-i$};
			
			\coordinate (c1) at (0.3333333,  0.4714045);
			\coordinate (c2) at (0.3333333, -0.4714045);
			
			\foreach \C in {c1,c2}{
				\draw[blue,thick] (\C) ++(-0.015,-0.015) -- ++(0.03,0.03);
				\draw[blue,thick] (\C) ++(-0.015, 0.015) -- ++(0.03,-0.03);
			}
			\node[blue,font=\small,anchor=west] at (c1) {$z_1'$};
			\node[blue,font=\small,anchor=west] at (c2) {$z_2'$};
			
			\draw[black,densely dashed,thick] (O) -- (z);
			\draw[black,densely dashed,thick] (O) -- (zp);
			
			\draw[black!65,thick] (c1) -- (z);
			\draw[black!65,thick] (c1) -- (zp);
			\draw[black!65,thick] (c2) -- (z);
			\draw[black!65,thick] (c2) -- (zp);
			
			\def\rO{0.18}
			\draw[magenta,thick] (\rO,0) arc[start angle=0,end angle=90,radius=\rO];
			\node[font=\scriptsize] at ($(O)+(0.22,0.12)$) {$\frac{\pi}{2}$};
			
			\pgfmathsetmacro{\aone}{atan2(0-0.4714045,1-0.3333333)}
			\pgfmathsetmacro{\bone}{atan2(1-0.4714045,0-0.3333333)}
			\def\rcone{0.10}
			\draw[magenta,thick]
			($(c1)+(\aone:\rcone)$) arc[start angle=\aone,end angle=\bone,radius=\rcone];
			\node[font=\scriptsize] at ($(c1)+(0.02,0.14)$) {$\frac{7\pi}{8}$};
			
			\pgfmathsetmacro{\atwo}{atan2(0+0.4714045,1-0.3333333)}
			\pgfmathsetmacro{\btwo}{atan2(1+0.4714045,0-0.3333333)}
			\def\rctwo{0.10}
			\draw[magenta,thick]
			($(c2)+(\atwo:\rctwo)$) arc[start angle=\atwo,end angle=\btwo,radius=\rctwo];
			\node[font=\scriptsize] at ($(c2)+(0.02,0.14)$) {$\frac{3\pi}{8}$};
			
		\end{tikzpicture}
		\caption{An illustration of Theorem~\ref{thm:Ge_Gonek} for $\mathcal{L}(u)=(u-1)(u-i)(u+i)$.}
		\label{fig:three-angles-pi}
	\end{figure}
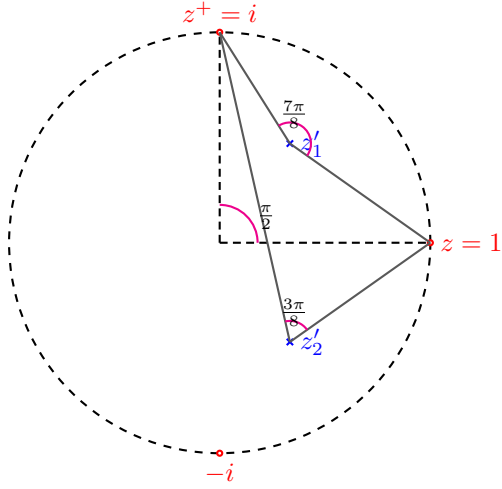
	
	Ge and Gonek \cite{GG24} also derived the following \emph{gap principle}: if there is a large gap between consecutive \emph{distinct} zeros of $\mathcal{L}$ on $\T$, then most critical points of $\mathcal{L}$ must lie close to  $\T$.
	
	\begin{thm}[Ge--Gonek]\label{thm:Ge_Gonek_gap}
		Let $\zeta_1,\dots,\zeta_M$ be the distinct zeros of $\mathcal{L}(u)$ with arguments
		\[
		0\le \theta_1<\theta_2<\cdots<\theta_M<2\pi,
		\qquad
		\theta_{M+1}:=\theta_1+2\pi,
		\]
		and let $G$ be the largest gap between arguments of consecutive \emph{distinct} zeros, i.e.,
		\[
		G=\max_{1\le j\le M}\bigl(\theta_{j+1}-\theta_j\bigr).
		\]
		Let $\mathcal{N}_\varepsilon$ be the number of critical points of $\mathcal{L}(u)$ at distance greater than $\varepsilon$ from the unit circle, i.e.,
		\[
		\mathcal{N}_\varepsilon=\bigl|\{\,z_k' : \abs{z_k'}<1-\varepsilon\,\}\bigr|.
		\]
		Then there exists an absolute constant $c_0>0$ such that
		\[
		\mathcal{N}_\varepsilon \le \frac{c_0}{\varepsilon\,G}.
		\]
		In particular, one may take $c_0=4\pi$.
		(Here ``absolute'' means independent of $N$, $\varepsilon$, and $G$.)
	\end{thm}
	Figure~\ref{fig:gap-principle-example} illustrates the case $N=3$ in Theorem~\ref{thm:Ge_Gonek_gap}.
	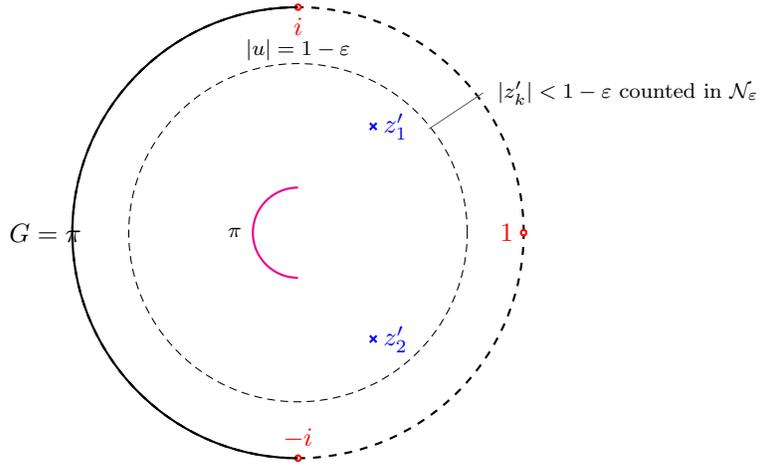
\begin{figure}[htbp]
		\centering
		\begin{tikzpicture}[scale=3.0]
			\coordinate (O) at (0,0);
			
			\def\reps{0.75}
			
			\draw[black,dashed,thick] (O) circle (1);
			\draw[black,densely dashed] (O) circle (\reps);
			\node[font=\scriptsize] at (0,\reps+0.06) {$\abs{u}=1-\varepsilon$};
			
			\coordinate (z0)   at (1,0);
			\coordinate (zpi2) at (0,1);
			\coordinate (z3pi2) at (0,-1);
			
			\foreach \P/\lab/\pos in {z0/$1$/east,zpi2/$i$/north,z3pi2/$-i$/south}{
				\draw[red,fill=white,thick] (\P) circle (0.012);
				\node[red,font=\small,anchor=\pos] at (\P) {\lab};
			}
			
			\draw[black,thick] (zpi2) arc[start angle=90,end angle=270,radius=1];
			\node[font=\small] at (-1.12,0) {$G=\pi$};
			
			\def\rG{0.20}
			\draw[magenta,thick] ($(O)+(0,\rG)$) arc[start angle=90,end angle=270,radius=\rG];
			\node[font=\scriptsize] at (-0.28,0) {$\pi$};
			
			\coordinate (c1) at (0.3333333,  0.4714045);
			\coordinate (c2) at (0.3333333, -0.4714045);
			
			\foreach \C/\lab/\pos in {c1/$z_1'$/west,c2/$z_2'$/west}{
				\draw[blue,thick] (\C) ++(-0.015,-0.015) -- ++(0.03,0.03);
				\draw[blue,thick] (\C) ++(-0.015, 0.015) -- ++(0.03,-0.03);
				\node[blue,font=\small,anchor=\pos] at (\C) {\lab};
			}
			
			\draw[black!60] (0.58,0.46) -- (0.82,0.62);
			\node[font=\scriptsize,anchor=west] at (0.84,0.62)
			{$\abs{z_k'}<1-\varepsilon$ counted in $\mathcal{N}_\varepsilon$};
			
		\end{tikzpicture}
		\caption{Illustration of Theorem~\ref{thm:Ge_Gonek_gap} for $\mathcal{L}(u)=(u-1)(u-i)(u+i)$.}
		\label{fig:gap-principle-example}
	\end{figure}
	
	D\'{\i}az-Barrero and Egozcue \cite{DE08} introduced \emph{incomplete polynomials} and showed that the derivative admits a representation as a convex combination of them, thereby providing a generalization of the notion of the derivative for complex polynomials. More precisely, let
	\[
	p(u)=\prod_{j=1}^N (u-z_j)
	\]
	be a monic polynomial of degree $N\ge 2$ (with zeros $z_j$ counted with multiplicity).
	Define the associated incomplete polynomials
	\[
	p_j(u)=\prod_{k\neq j}(u-z_k)=\frac{p(u)}{u-z_j},
	\qquad 1\le j\le N.
	\]
	Given weights $\lambda_1,\dots,\lambda_N\ge 0$ with $\sum_{j=1}^N\lambda_j=1$, define
	\begin{equation*}
		p_\lambda(u)=\sum_{j=1}^N \lambda_j\, p_j(u).
	\end{equation*}
	Since each $p_j$ is monic of degree $N-1$, the polynomial $p_\lambda$ is also monic of degree $N-1$.
	In particular, the choice $\lambda_1=\cdots=\lambda_N=1/N$ gives
	\[
	\frac{1}{N}\sum_{j=1}^N p_j(u)=\frac{p'(u)}{N}.
	\]
Many classical results on the geometry of zeros and critical points extend from the ordinary derivative $p'$ to the more general setting of zeros of $p$ and of convex combinations of its incomplete polynomials.  Examples include Gauss--Lucas type theorems \cite{DE08,Zha24}, Laguerre's theorem for polar derivatives \cite{KD22}, and various majorization relations \cite{Zha24}.

	In this paper, we apply this construction to $p=\mathcal{L}$ and write
	\[
	\mathcal{L}_\lambda(u):=\sum_{j=1}^N \lambda_j\,\frac{\mathcal{L}(u)}{u-z_j}.
	\]
	For convenience, set
	\[
	\mathcal{L}_j(u):=\frac{\mathcal{L}(u)}{u-z_j},
	\qquad 1\le j\le N,
	\]
	so that $\mathcal{L}_\lambda(u)=\sum_{j=1}^N \lambda_j\,\mathcal{L}_j(u)$.
	Our goal is to extend Theorem~\ref{thm:Ge_Gonek} and Theorem~\ref{thm:Ge_Gonek_gap} from critical points of $\mathcal{L}$ to zeros of $\mathcal{L}_\lambda$.
	For the angle duality, we will assume \emph{strictly positive} weights; otherwise the conclusion may fail (see Remark~\ref{rem:zero-weight-counterexample} below).

	\begin{thm}\label{thm:incomplete_dual}
		Assume that $\lambda_1,\dots,\lambda_N>0$ and $\sum_{j=1}^N\lambda_j=1$.
		Let $z=e^{i\theta}$ and $z^{+}=e^{i\theta^{+}}$ be two consecutive \emph{distinct} zeros of $\mathcal{L}(u)$.
		Let $w_1,\ldots,w_{N-1}$ be the zeros of $\mathcal{L}_\lambda(u)$, counted with multiplicity.
		Then
		\begin{equation}\label{eq:main-angle}
			\sum_{k=1}^{N-1}\Theta(w_k;z,z^{+})
			=\pi+(N-2)\frac{\theta^{+}-\theta}{2},
		\end{equation}
		where angles at endpoints are interpreted using Definition~\ref{def:endpoint} when needed.
	\end{thm}
	
	\begin{rem}[A zero-weight counterexample]\label{rem:zero-weight-counterexample}
		The positivity assumption $\lambda_j>0$ in Theorem~\ref{thm:incomplete_dual} is essential.
		If one allows some weights to vanish, then \eqref{eq:main-angle} may fail, even if one interprets angles at endpoints using Definition~\ref{def:endpoint}.
		
		For example, let
		\[
		\mathcal{L}(u)=(u-1)(u-i)(u+i),
		\qquad
		(\lambda_1,\lambda_2,\lambda_3)=\Bigl(0,\frac12,\frac12\Bigr).
		\]
		Then
		\[
		\mathcal{L}_\lambda(u)
		=\frac12\frac{\mathcal{L}(u)}{u-i}+\frac12\frac{\mathcal{L}(u)}{u+i}
		=\frac12(u-1)(u+i)+\frac12(u-1)(u-i)
		=u(u-1),
		\]
		so the zeros of $\mathcal{L}_\lambda$ are $w_1=0$ and $w_2=1$.
		Choose the consecutive zeros $z=1$ and $z^{+}=i$ of $\mathcal{L}$, so that $\theta^{+}-\theta=\pi/2$.
		Then
		\[
		\Theta(0;1,i)=\frac{\pi}{2},
		\qquad
		\Theta(1;1,i)=\frac{\pi}{4}
		\quad\text{(by Definition~\ref{def:endpoint})},
		\]
		whence
		\[
		\Theta(w_1;1,i)+\Theta(w_2;1,i)=\frac{3\pi}{4},
		\]
		while the right-hand side of \eqref{eq:main-angle} equals
		\[
		\pi+(3-2)\frac{\pi/2}{2}=\pi+\frac{\pi}{4}=\frac{5\pi}{4}.
		\]
	\end{rem}
	
	\begin{thm}\label{thm:incomplete_gap}
		Assume that $\lambda_1,\dots,\lambda_N>0$ and $\sum_{j=1}^N\lambda_j=1$.
		Let $\zeta_1,\dots,\zeta_M$ be the distinct zeros of $\mathcal{L}(u)$ with arguments
		\[
		0\le \theta_1<\theta_2<\cdots<\theta_M<2\pi,
		\qquad
		\theta_{M+1}:=\theta_1+2\pi,
		\]
		and let $G$ be the largest gap between arguments of consecutive \emph{distinct} zeros of $\mathcal{L}(u)$, i.e.,
		\[
		G=\max_{1\le j\le M}\bigl(\theta_{j+1}-\theta_j\bigr).
		\]
		Let $\mathcal{N}_\varepsilon$ be the number of zeros of $\mathcal{L}_\lambda(u)$ at distance greater than $\varepsilon$ from the unit circle, i.e.,
		\begin{equation}\label{eq:def-Neps}
			\mathcal{N}_\varepsilon
			=\bigl|\{\,w_k:\ \abs{w_k}<1-\varepsilon\,\}\bigr|,
		\end{equation}
		counted with multiplicity.
		Then there exists an absolute constant $c_0>0$ such that
		\[
		\mathcal{N}_\varepsilon \le \frac{c_0}{\varepsilon\,G}.
		\]
		In particular, one may take $c_0=4\pi$.
		(Here ``absolute'' means independent of $N$, $\varepsilon$, $G$, and $\lambda$.)
	\end{thm}
		
	\begin{rem}\label{rem:sendov-fails}
		Ge and Gonek \cite{GG24} used their angle-duality identity to strengthen Sendov's conjecture for polynomials whose zeros lie on the unit circle.
		However, a direct Sendov-type analogue fails for convex combinations of incomplete polynomials (see Example~\ref{eg} below). 	
	\end{rem}
	More precisely, even when all zeros of $\mathcal{L}$ lie on $\T$ and all weights are strictly positive,
		it is \emph{not} true in general that for every zero $z$ of $\mathcal{L}$ there must exist a zero $w$ of $\mathcal{L}_\lambda$ with
		\[
		|z-w|\le 1.
		\]
		Here is an explicit example.
\begin{eg}\label{eg}
		Let
	\[
	\mathcal{L}(u)=u^3-1=(u-1)(u-\omega)(u-\omega^2),
	\qquad \omega=e^{2\pi i/3},
	\]
	and choose weights
	\[
	(\lambda_1,\lambda_2,\lambda_3)=\Bigl(\frac45,\frac1{10},\frac1{10}\Bigr).
	\]
	Then a direct computation gives
	\[
	\mathcal{L}_\lambda(u)
	=\sum_{j=1}^3 \lambda_j\,\frac{\mathcal{L}(u)}{u-z_j}
	= u^2+\frac{7}{10}u+\frac{7}{10}.
	\]
	Hence the zeros of $\mathcal{L}_\lambda$ are
	\[
	w_{\pm}=-\frac{7}{20}\pm i\,\frac{\sqrt{231}}{20}.
	\]
	For the zero $z=1$ of $\mathcal{L}$ we have
	\[
	|1-w_{\pm}|^2=\Bigl(1+\frac{7}{20}\Bigr)^2+\Bigl(\frac{\sqrt{231}}{20}\Bigr)^2
	=\frac{12}{5},
	\qquad\text{so}\qquad |1-w_{\pm}|=\sqrt{\frac{12}{5}}>1.
	\]
	Thus \emph{neither} zero of $\mathcal{L}_\lambda$ lies in the closed unit disk centered at $z=1$,
	showing that the Sendov-type radius-$1$ conclusion does not extend to $\mathcal{L}_\lambda$.
\end{eg}		
	
		\medskip
\noindent\textbf{Organization of the paper.}
In Section~\ref{sec:Preliminaries}, we introduce the $\Dshape$-region and collect several elementary geometric facts about subtended angles and argument differences.
In Section~\ref{sec:Proof_of_thm_incomplete_dual}, we prove the angle duality for convex combinations of incomplete polynomials (Theorem~\ref{thm:incomplete_dual}) by an argument-variation identity and a continuity argument to remove the $2\pi$ ambiguity.
In Section~\ref{sec:Two geometric lemmas for the gap principle}, we establish two geometric lemmas, including a quantitative ``angle gain'' away from the unit circle.
Finally, in Section~\ref{sec:Proof of Theorem incomplete_gap}, we combine the angle duality with the geometric estimates to prove the gap principle for $\mathcal{L}_\lambda$ (Theorem~\ref{thm:incomplete_gap}).

	\section{Preliminaries}\label{sec:Preliminaries}
	
	\begin{defin}[$\Dshape$-region]\label{def:D-region}
		Let $z=e^{i\theta}$ and $z^{+}=e^{i\theta^{+}}$ be distinct points on $\T$ with $\theta<\theta^{+}<\theta+2\pi$.
		Let $H(z,z^{+})$ be the closed half-plane bounded by the line through $z$ and $z^{+}$ which contains the counterclockwise arc on $\T$ from $z^{+}$ to $z$ (including endpoints).
		Define
		\[
		\Dshape(z,z^{+}) := \udisc \cap H(z,z^{+}),
		\]see Figure~\ref{fig:D-region}.
	\end{defin}
	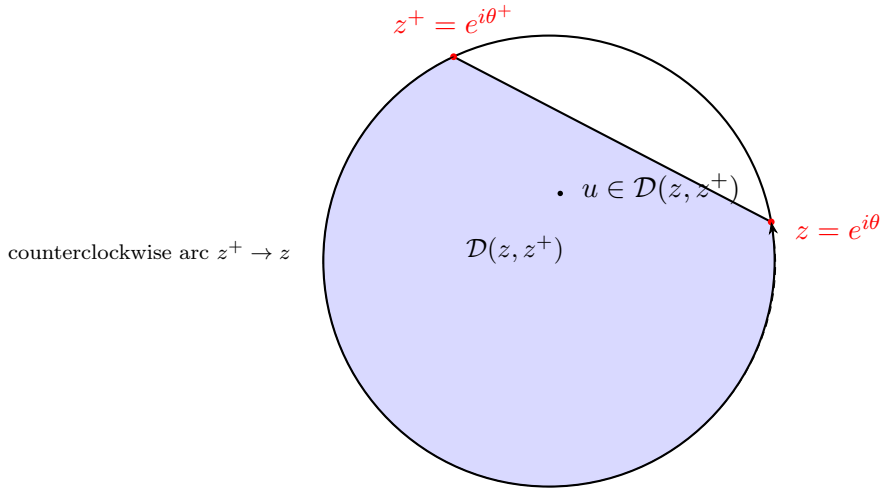
\begin{figure}[htbp]
		\centering
		\begin{tikzpicture}[scale=3,>=Stealth]
			
			\coordinate (O) at (0,0);
			
			\def\th{10}      
			\def\thp{115}    
			
			\coordinate (z)  at (\th:1);
			\coordinate (zp) at (\thp:1);
			
			\path[fill=blue!15]
			(z) -- (zp)
			arc[start angle=\thp, end angle=\th+360, radius=1]
			-- cycle;
			
			\draw[thick] (O) circle (1);
			\draw[thick] (z) -- (zp);
			
			\fill[red] (z)  circle (0.015);
			\fill[red] (zp) circle (0.015);
			
			\node[red,anchor=west]  at ($(z)+(0.06,-0.02)$) {$z=e^{i\theta}$};
			\node[red,anchor=south] at ($(zp)+(0.00,0.06)$) {$z^{+}=e^{i\theta^{+}}$};
			
			\draw[densely dashed,->]
			(\thp:1) arc[start angle=\thp, end angle=\th+360, radius=1];
			\node[font=\scriptsize,anchor=east] at (-1.10,0.05) {counterclockwise arc $z^{+}\to z$};
			
			\node[font=\small] at (-0.15,0.05) {$\Dshape(z,z^{+})$};
			
			\coordinate (u) at (0.05,0.30);
			\fill (u) circle (0.012);
			\node[anchor=west] at ($(u)+(0.05,0.02)$) {$u\in\Dshape(z,z^{+})$};
			
		\end{tikzpicture}
		\caption{A schematic illustration of the $\Dshape$-region $\Dshape(z,z^{+})=\udisc\cap H(z,z^{+})$.}
		\label{fig:D-region}
	\end{figure}

	A key ingredient is that the zeros of $p_\lambda$ satisfy the same convex-hull confinement as critical points; see \cite[Theorem~1.2]{DE08} or \cite[Theorem~1.3]{Zha24}. Let $\conv\{z_1,\dots,z_N\}$ denote the convex hull of $z_1,\dots,z_N$.

	\begin{lem}[\cite{DE08,Zha24}]\label{lem:q-conv}
		Let $p(u)=\prod_{j=1}^N(u-z_j)$ be monic and let
	$	p_\lambda(u)=\sum_{j=1}^N\lambda_j\,p_j(u)$
		with $\lambda_j\ge 0$ and $\sum_{j=1}^N\lambda_j=1$.
		Then every zero of $p_\lambda$ lies in $\conv\{z_1,\dots,z_N\}$.
		In particular, every zero of $\mathcal{L}_\lambda$ lies in $\udisc$.
	\end{lem}
As an immediate corollary of Lemma~\ref{lem:q-conv}, we obtain
	\begin{cor}\label{cor:q-in-D}
		Let $z=e^{i\theta}$ and $z^{+}=e^{i\theta^{+}}$ be consecutive \emph{distinct} zeros of $\mathcal{L}$.
		Then every zero of $\mathcal{L}_\lambda$ lies in $\Dshape(z,z^{+})$.
	\end{cor}

	\begin{lem}[Argument difference equals subtended angle]\label{lem:arg-angle}
		Let $z,z^{+}\in\C$ be distinct and let $w\in\C\setminus\{z,z^{+}\}$.
		Then
		\[
		\arg(z^{+}-w)-\arg(z-w)\equiv \pm\,\Theta(w;z,z^{+})\pmod{2\pi}.
		\]
		Moreover, if $z,z^{+}\in\T$ and $w\in \udisc\cap H(z,z^{+})$ (in particular, if $w\in\Dshape(z,z^{+})$), then the unique representative of $\arg(z^{+}-w)-\arg(z-w)$ in $(0,\pi]$ equals $\Theta(w;z,z^{+})$.
		If instead $w\in \udisc\setminus H(z,z^{+})$, then the unique representative in $[-\pi,0)$ equals $-\Theta(w;z,z^{+})$ (equivalently, the representative in $(\pi,2\pi)$ equals $2\pi-\Theta(w;z,z^{+})$).
	\end{lem}
	\begin{proof}
		The difference $\arg(z^{+}-w)-\arg(z-w)$ is the oriented angle from the ray from $w$ to $z$ to the ray from $w$ to $z^{+}$, hence it is congruent to $\pm\,\Theta(w;z,z^{+})$ modulo $2\pi$.
		
		Assume now that $z,z^{+}\in\T$.
		The line through $z$ and $z^{+}$ divides the plane into two open half-planes.
		On each open half-plane, the sign of the oriented angle from the ray from $w$ to $z$ to the ray from $w$ to $z^{+}$ is constant (it can only change when $w$ crosses the line through $z$ and $z^{+}$).
		By definition, $H(z,z^{+})$ is the \emph{closed} half-plane containing the counterclockwise arc from $z^{+}$ to $z$, so for $w$ in the interior of $H(z,z^{+})$ the ray from $w$ to $z^{+}$ is obtained from the ray from $w$ to $z$ by a counterclockwise rotation; thus the oriented angle lies in $(0,\pi]$ and its representative in $(0,\pi]$ equals $\Theta(w;z,z^{+})$.
		For $w$ in the opposite open half-plane, the rotation is clockwise, giving the stated representatives in $[-\pi,0)$ and $(\pi,2\pi)$.
		
		If, in addition, $w\in\udisc$ lies on the line through $z$ and $z^{+}$, then necessarily $w\in [z,z^{+}]$.
		In this case the vectors $z-w$ and $z^{+}-w$ are opposite, so the representative of $\arg(z^{+}-w)-\arg(z-w)$ in $(0,\pi]$ equals $\pi$, matching $\Theta(w;z,z^{+})=\pi$.
	\end{proof}
	
	\begin{lem}[Half-gap on the unit circle]\label{lem:half-gap}
		Let $z=e^{i\theta}$ and $z^{+}=e^{i\theta^{+}}$ with $\theta<\theta^{+}<\theta+2\pi$.
		If $w=e^{it}$ lies on the unit circle, $w\neq z,z^{+}$, and $w$ is not on the open arc from $z$ to $z^{+}$ (counterclockwise), then
		\[
		\Theta(w;z,z^{+})=\frac{\theta^{+}-\theta}{2}.
		\]
	\end{lem}
	
	\begin{proof}
		This is the inscribed angle theorem: the inscribed angle subtending the chord $[z,z^{+}]$ equals one half of the measure of the opposite arc.
		Under the hypothesis, the opposite arc is precisely the counterclockwise arc from $z$ to $z^{+}$ of length $\theta^{+}-\theta$.
	\end{proof}
	
	\section{Proof of Theorem~\ref{thm:incomplete_dual}} \label{sec:Proof_of_thm_incomplete_dual}
	\begin{proof}[Proof of Theorem~\ref{thm:incomplete_dual}]
		\textit{Step 1: the case of pairwise distinct zeros.}
		Assume first that $z_1,\dots,z_N$ are pairwise distinct. Since the zeros are distinct, we may (and do) relabel them so that
		\[
		0\le \arg z_1<\arg z_2<\cdots<\arg z_N<2\pi,
		\qquad
		\arg z_{N+1}:=\arg z_1+2\pi,
		\]
		and we use cyclic indexing. In particular, for a consecutive pair we write
		$z=z_j$ and $z^{+}=z_{j+1}$.
		
		Let $z=z_j$ and $z^{+}=z_{j+1}$ be consecutive.
		Write
		\[
		\mathcal{L}_\lambda(u)=\prod_{k=1}^{N-1}(u-w_k).
		\]
		Then
		\begin{equation}\label{eq:ratio-zeros}
			\frac{\mathcal{L}_\lambda(z^{+})}{\mathcal{L}_\lambda(z)}=\prod_{k=1}^{N-1}\frac{z^{+}-w_k}{z-w_k}.
		\end{equation}
		
		At the (simple) zero $z_j$ of $\mathcal{L}$ we have $\mathcal{L}_\ell(z_j)=0$ for $\ell\neq j$, while $\mathcal{L}_j(z_j)\neq 0$; hence
		\[
		\mathcal{L}_\lambda(z)=\lambda_j \mathcal{L}_j(z),\qquad \mathcal{L}_\lambda(z^{+})=\lambda_{j+1}\mathcal{L}_{j+1}(z^{+}).
		\]
		Therefore,
		\[
		\frac{\mathcal{L}_\lambda(z^{+})}{\mathcal{L}_\lambda(z)}
		=\frac{\lambda_{j+1}}{\lambda_j}\cdot\frac{\mathcal{L}_{j+1}(z^{+})}{\mathcal{L}_j(z)}
		=-\frac{\lambda_{j+1}}{\lambda_j}\prod_{\ell\neq j,j+1}\frac{z^{+}-z_\ell}{z-z_\ell}.
		\]
		Combining with \eqref{eq:ratio-zeros} gives
		\begin{equation}\label{eq:prod-id}
			\prod_{k=1}^{N-1}\frac{z^{+}-w_k}{z-w_k}
			=-\frac{\lambda_{j+1}}{\lambda_j}\prod_{\ell\neq j,j+1}\frac{z^{+}-z_\ell}{z-z_\ell}.
		\end{equation}
		
		Taking arguments on both sides of \eqref{eq:prod-id} yields a congruence modulo $2\pi$.
		Since $\lambda_{j+1}/\lambda_j>0$ contributes no argument and $-1$ contributes $\pi$, we obtain
		\begin{equation}\label{eq:arg-cong}
			\sum_{k=1}^{N-1}\Big(\arg(z^{+}-w_k)-\arg(z-w_k)\Big)
			\equiv
			\pi+\sum_{\ell\neq j,j+1}\Big(\arg(z^{+}-z_\ell)-\arg(z-z_\ell)\Big)
			\pmod{2\pi}.
		\end{equation}
		
		By Corollary~\ref{cor:q-in-D}, each $w_k\in \Dshape(z,z^{+})=\udisc\cap H(z,z^{+})$; hence by Lemma~\ref{lem:arg-angle}, the unique representative of each left-hand summand in $(0,\pi]$ equals $\Theta(w_k;z,z^{+})$.
		On the right-hand side, each $z_\ell$ lies on the unit circle and, because $z$ and $z^{+}$ are consecutive, lies outside the open arc from $z$ to $z^{+}$; thus Lemma~\ref{lem:half-gap} gives $\Theta(z_\ell;z,z^{+})=(\theta^{+}-\theta)/2$ for all $\ell\neq j,j+1$.
		Therefore \eqref{eq:arg-cong} implies
		\begin{equation}\label{eq:cong-final}
			\sum_{k=1}^{N-1}\Theta(w_k;z,z^{+})
			\equiv
			\pi+(N-2)\frac{\theta^{+}-\theta}{2}
			\pmod{2\pi}.
		\end{equation}
		
		It remains to remove the ambiguity modulo $2\pi$.
		By rotating the variable and relabeling indices cyclically, we may assume $z=z_1=1$ and $z^{+}=z_2$.
		Consider the connected parameter space
		\[
		\Omega=\Big\{(\boldsymbol{\theta},\lambda):\ 0=\theta_1<\theta_2<\cdots<\theta_N<2\pi,\ \lambda\in\Delta^\circ\Big\},
		\]
		where $\boldsymbol{\theta}=(\theta_1,\dots,\theta_N)$ and
		\[
		\Delta^\circ=\bigl\{(\lambda_1,\dots,\lambda_N)\in(0,1)^N:\ \sum_{j=1}^N\lambda_j=1\bigr\}.
		\]
		For each point of $\Omega$ define $z_j=e^{i\theta_j}$, the polynomial $\mathcal{L}(u)=\prod_{j=1}^N(u-z_j)$, and $\mathcal{L}_\lambda(u)=\sum_{j=1}^N\lambda_j\mathcal{L}(u)/(u-z_j)$, with consecutive pair $(z,z^{+})=(z_1,z_2)$.
		Set
		\[
		\Phi(\boldsymbol{\theta},\lambda):=
		\sum_{k=1}^{N-1}\Theta(w_k;z,z^{+})
		-\Big(\pi+(N-2)\frac{\theta_2-\theta_1}{2}\Big).
		\]
		By \eqref{eq:cong-final}, $\Phi(\boldsymbol{\theta},\lambda)\in 2\pi\mathbb{Z}$ for all $(\boldsymbol{\theta},\lambda)\in\Omega$.
		
		Since $\lambda_j>0$ and the zeros $z_j$ are distinct, we have $\mathcal{L}_\lambda(z_1)\neq 0$ and $\mathcal{L}_\lambda(z_2)\neq 0$, so none of the zeros $w_k$ coincides with $z$ or $z^{+}$.
		As the coefficients of $\mathcal{L}_\lambda$ depend continuously on $(\boldsymbol{\theta},\lambda)$, it is well known that the roots $\{w_1,\dots,w_{N-1}\}$ (counting multiplicity) depend continuously on $(\boldsymbol{\theta},\lambda)$ as an unordered multiset.
		Since the function $w\mapsto \Theta(w;z,z^{+})$ is continuous on $\C\setminus\{z,z^{+}\}$, the symmetric sum $\sum_{k=1}^{N-1}\Theta(w_k;z,z^{+})$ varies continuously with the multiset of roots.
		Consequently, $\Phi$ is continuous on $\Omega$.
		Since $\Omega$ is connected and $2\pi\mathbb{Z}$ is discrete, $\Phi$ is constant on $\Omega$.
		
		To identify the constant, take $\theta_j=2\pi(j-1)/N$ and $\lambda_j=1/N$.
		Then $\mathcal{L}(u)=u^N-1$ and $\mathcal{L}_\lambda(u)=\mathcal{L}'(u)/N=u^{N-1}$, so $w_1=\cdots=w_{N-1}=0$.
		For $z=1$ and $z^{+}=e^{2\pi i/N}$ we have $\Theta(0;z,z^{+})=2\pi/N$, hence
		\[
		\sum_{k=1}^{N-1}\Theta(w_k;z,z^{+})=(N-1)\frac{2\pi}{N}=2\pi-\frac{2\pi}{N},
		\]
		while
		\[
		\pi+(N-2)\frac{\theta^{+}-\theta}{2}
		=\pi+(N-2)\frac{\pi}{N}=2\pi-\frac{2\pi}{N}.
		\]
		Thus $\Phi=0$ at this point, and hence $\Phi\equiv 0$ on $\Omega$.
		This proves \eqref{eq:main-angle} when the zeros of $\mathcal{L}$ are pairwise distinct.
		
		\medskip
		\textit{Step 2: allowing multiple zeros.}
		Now let $\mathcal{L}$ have (possibly) multiple zeros on $\T$, and let $z=e^{i\theta}$ and $z^{+}=e^{i\theta^{+}}$ be a consecutive \emph{distinct} pair of zeros, so $\theta<\theta^{+}<\theta+2\pi$.
		Let $\zeta_1,\dots,\zeta_M$ be the distinct zeros of $\mathcal{L}$, and let $m_1,\dots,m_M$ be their multiplicities, so that $\sum_{r=1}^M m_r=N$.
		Write $z=\zeta_j$ and $z^{+}=\zeta_{j+1}$ (cyclic indexing), and set $\alpha:=\theta^{+}-\theta$.
		
		Group the weights by distinct zeros:
		\[
		\Lambda_r:=\sum_{\ell:\ z_\ell=\zeta_r}\lambda_\ell\qquad (1\le r\le M).
		\]
		Then $\Lambda_r>0$ for all $r$ and $\sum_{r=1}^M\Lambda_r=1$.
		Define
		\[
		\widetilde{\mathcal{L}}(u):=\prod_{r=1}^M (u-\zeta_r),
		\qquad
		Q(u):=\prod_{r=1}^M (u-\zeta_r)^{m_r-1},
		\]
		so that $\mathcal{L}(u)=Q(u)\,\widetilde{\mathcal{L}}(u)$.
		For each $r$, set $\widetilde{\mathcal{L}}_r(u):=\widetilde{\mathcal{L}}(u)/(u-\zeta_r)$.
		If $z_\ell=\zeta_r$, then
		\[
		\mathcal{L}_\ell(u)=\frac{\mathcal{L}(u)}{u-z_\ell}=\frac{\mathcal{L}(u)}{u-\zeta_r}=Q(u)\,\widetilde{\mathcal{L}}_r(u).
		\]
		Hence
		\[
		\mathcal{L}_\lambda(u)=\sum_{\ell=1}^N\lambda_\ell\,\mathcal{L}_\ell(u)
		=\sum_{r=1}^M \Lambda_r\, Q(u)\,\widetilde{\mathcal{L}}_r(u)
		=Q(u)\,\widetilde{\mathcal{L}}_\Lambda(u),
		\]
		where
		\[
		\widetilde{\mathcal{L}}_\Lambda(u):=\sum_{r=1}^M \Lambda_r\,\widetilde{\mathcal{L}}_r(u).
		\]
		Let $\widetilde{w}_1,\dots,\widetilde{w}_{M-1}$ be the zeros of $\widetilde{\mathcal{L}}_\Lambda$, counted with multiplicity.
		Then the zeros of $\mathcal{L}_\lambda$ consist of:
		(i) each $\zeta_r$ with multiplicity $m_r-1$ (coming from $Q$), and
		(ii) the zeros $\widetilde{w}_1,\dots,\widetilde{w}_{M-1}$.
		
		Therefore,
		\[
		\sum_{k=1}^{N-1}\Theta(w_k;z,z^{+})
		=
		\sum_{r=1}^M (m_r-1)\Theta(\zeta_r;z,z^{+})
		+\sum_{k=1}^{M-1}\Theta(\widetilde{w}_k;z,z^{+}).
		\]
		Since $z$ and $z^{+}$ are consecutive distinct zeros, every $\zeta_r$ lies outside the open arc from $z$ to $z^{+}$.
		Thus Lemma~\ref{lem:half-gap} gives $\Theta(\zeta_r;z,z^{+})=\alpha/2$ for $\zeta_r\notin\{z,z^{+}\}$, and Definition~\ref{def:endpoint} gives the same value when $\zeta_r\in\{z,z^{+}\}$.
		Consequently,
		\[
		\sum_{r=1}^M (m_r-1)\Theta(\zeta_r;z,z^{+})=(N-M)\frac{\alpha}{2}.
		\]
		
		Finally, $\widetilde{\mathcal{L}}$ has pairwise distinct zeros on $\T$ and the weights $\Lambda_r$ are strictly positive, so Step~1 applied to $\widetilde{\mathcal{L}}$ and $\widetilde{\mathcal{L}}_\Lambda$ yields
		\[
		\sum_{k=1}^{M-1}\Theta(\widetilde{w}_k;z,z^{+})
		=\pi+(M-2)\frac{\alpha}{2}.
		\]
		Adding the contributions gives
		\[
		\sum_{k=1}^{N-1}\Theta(w_k;z,z^{+})
		=(N-M)\frac{\alpha}{2}+\pi+(M-2)\frac{\alpha}{2}
		=\pi+(N-2)\frac{\alpha}{2},
		\]
		which is exactly \eqref{eq:main-angle}.
	\end{proof}

	\section{Two geometric lemmas for the gap principle} \label{sec:Two geometric lemmas for the gap principle}
	
	\begin{lem}[A diameter bound in the $\Dshape$-region]\label{lem:D-diameter}
		Let $z=e^{i\theta}$ and $z^{+}=e^{i\theta^{+}}$ with $\alpha:=\theta^{+}-\theta\in(\pi,2\pi)$.
		Then for every $u\in \Dshape(z,z^{+})$,
		\[
		\abs{u-z}\le \abs{z^{+}-z}=2\sin\!\Big(\frac{\alpha}{2}\Big).
		\]
	\end{lem}
	
	\begin{proof}
		Since $\Dshape(z,z^{+})=\udisc\cap H(z,z^{+})$ is an intersection of convex sets, it is convex and compact.
		
		Fix $u\in\Dshape(z,z^{+})$, $u\neq z$.
		Consider the ray $\{z+t(u-z):t\ge 0\}$.
		By compactness there exists $t_*\ge 1$ such that $u_*:=z+t_*(u-z)\in\partial\Dshape(z,z^{+})$ and $[z,u_*]\subset\Dshape(z,z^{+})$.
		Then $u=z+s(u_*-z)$ with $s=1/t_*\in(0,1]$, hence $\abs{u-z}=s\,\abs{u_*-z}\le \abs{u_*-z}$.
		So it suffices to prove the bound for $u_*\in\partial\Dshape(z,z^{+})$.
		
		If $u_*$ lies on the chord $[z,z^{+}]$, then $\abs{u_*-z}\le \abs{z^{+}-z}$ is immediate.
		If instead $u_*=e^{it}$ lies on the boundary arc from $z^{+}$ to $z$, then the shorter central angle between $u_*$ and $z$ is at most $2\pi-\alpha\in(0,\pi)$, hence
		\[
		\abs{u_*-z}=\abs{e^{it}-e^{i\theta}}
		=2\sin\!\Big(\frac{\delta}{2}\Big)
		\le 2\sin\!\Big(\frac{2\pi-\alpha}{2}\Big)
		=2\sin\!\Big(\frac{\alpha}{2}\Big)
		=\abs{z^{+}-z},
		\]
		where $\delta\in[0,2\pi-\alpha]$ is that shorter central angle.
		(Note that $\sin(\alpha/2)=\sin((2\pi-\alpha)/2)$, so this equals the usual chord length.)
	\end{proof}
	
\begin{lem}[Angle gain away from the unit circle]\label{lem:angle-gain}

		Let $z=e^{i\theta}$ and $z^{+}=e^{i\theta^{+}}$ be distinct points on $\T$
		with $\theta<\theta^{+}<\theta+2\pi$, and set $\alpha:=\theta^{+}-\theta\in(0,2\pi)$.
		Fix $\varepsilon\in(0,1)$.
		Then for any $u\in \Dshape(z,z^{+})$ with $\abs{u}<1-\varepsilon$,
	\[
	\Theta(u;z,z^{+})-\frac{\alpha}{2}\ \ge\
	\begin{cases}
		\displaystyle \frac{\varepsilon}{2}\,\sin\!\Big(\frac{\alpha}{2}\Big),
		& \text{if }\alpha\le \pi,\\[10pt]
		\displaystyle \frac{\varepsilon}{2},
		& \text{if }\alpha>\pi.
	\end{cases}
	\]
\end{lem}
\begin{figure}[t]
	\centering
	\begin{tikzpicture}[scale=3.1, line cap=round, line join=round, >=Stealth]
		
		\def\alphaDeg{60}   
		\def\rinner{0.75}   
		\def\ut{0.45}       
		
		\coordinate (O)  at (0,0);
		\coordinate (z)  at (0:1);
		\coordinate (zp) at (\alphaDeg:1);
		
		\coordinate (A)  at (210:1);
		
		\coordinate (u) at ($(A)!\ut!(zp)$);
		
		\tikzset{
			main/.style={thick},
			thinmain/.style={semithick},
			inner/.style={dotted, line width=0.35pt},
			help/.style={densely dashed},
			lab/.style={font=\small},
			ang/.style={draw, line width=0.6pt},
		}
		
		\draw[thinmain] (O) circle (1);
		\draw[inner]    (O) circle (\rinner);
		\node[lab] at (0,\rinner+0.06) {$|u|=1-\varepsilon$};
		
		\draw[main] (z) -- (zp);
		
		\draw[help,->] (z) arc[start angle=0, end angle=\alphaDeg, radius=1];
		\node[lab] at ({\alphaDeg/2}:1.14) {$\alpha$};
		
		\draw[main] (A) -- (zp);
		
		\draw[main] (A) -- (z);
		\draw[main] (u) -- (z);
		
		\draw[main] (A) -- (u)
		node[pos=0.55, sloped, above, fill=white, inner sep=1.2pt, lab]
		{$|Au|\ge \varepsilon$};
		
		\fill (z)  circle (0.015);
		\fill (zp) circle (0.015);
		\fill (A)  circle (0.015);
		\fill (u)  circle (0.015);
		
		\node[lab, anchor=west]  at ($(z)+(0.03,-0.02)$) {$z$};
		\node[lab, anchor=south] at ($(zp)+(0,0.03)$) {$z^{+}$};
		\node[lab, anchor=north east] at ($(A)+(-0.02,-0.02)$) {$A$};
		\node[lab, anchor=west]  at ($(u)+(0.02,0)$) {$u$};
		
		\pic[ang, angle radius=0.22, angle eccentricity=1.55,
		pic text={$\alpha/2$},
		pic text options={fill=white, inner sep=1pt, lab}
		] {angle = z--A--zp};
		
		\pic[ang, angle radius=0.18, angle eccentricity=1.85,
		pic text={$\phi_2$},
		pic text options={fill=white, inner sep=1pt, lab}
		] {angle = A--z--u};
		
	\end{tikzpicture}
	\caption{Auxiliary configuration in proof of Lemma~\ref{lem:angle-gain}.}
	\label{fig:angle-gain-geom}
\end{figure}
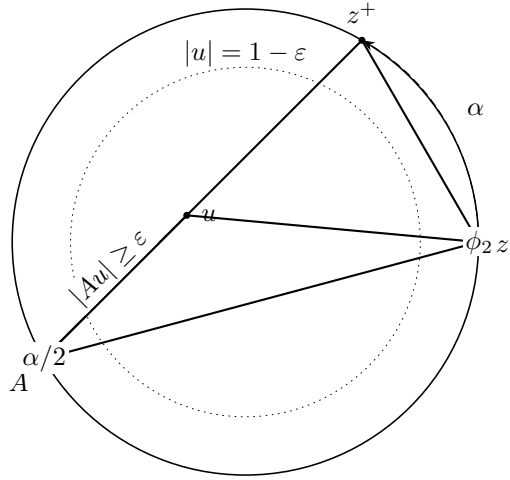

\begin{proof}
	If $u,z,z^{+}$ are collinear, then $u$ lies on the chord $[z,z^{+}]$ (since $u\in\udisc$ and the line through $z$ and $z^{+}$ meets $\T$ only at $z,z^{+}$), and hence $\Theta(u;z,z^{+})=\pi$.
	If $\alpha\le \pi$, then the claimed inequality is immediate since the right-hand side is at most $\varepsilon/2<\pi/2\le \pi-\alpha/2$.
	If $\alpha>\pi$, set $\beta:=2\pi-\alpha\in(0,\pi)$ (the shorter central angle between $z$ and $z^{+}$).
	After a rotation we may assume $z=e^{-i\beta/2}$ and $z^{+}=e^{i\beta/2}$, so the chord $[z,z^{+}]$ is the vertical segment with real part $\cos(\beta/2)$; in particular, every point on the chord satisfies $\abs{u}\ge \cos(\beta/2)$.
	Since $\abs{u}<1-\varepsilon$, we have $\cos(\beta/2)<1-\varepsilon$, so $1-\cos(\beta/2)>\varepsilon$.
	Using $1-\cos x\le x^2/2$ for all real $x$, we obtain $(\beta/2)^2/2>\varepsilon$, hence $\beta/2>\sqrt{2\varepsilon}\ge \varepsilon/2$ (as $\varepsilon\in(0,1)$).
	Therefore
	\[
	\Theta(u;z,z^{+})-\frac{\alpha}{2}
	=\pi-\frac{\alpha}{2}
	=\frac{\beta}{2}
	\ge \frac{\varepsilon}{2},
	\]
	which matches the $\alpha>\pi$ case.
	
	So assume $u,z,z^{+}$ are not collinear.
	
	Let the line through $u$ and $z^{+}$ meet the unit circle again at $A\neq z^{+}$.
	Since $H(z,z^{+})$ is a (closed) half-plane and $z^{+}\in\partial H(z,z^{+})$, the entire ray $\{\,z^{+}+t(u-z^{+}): t\ge 0\,\}$ is contained in $H(z,z^{+})$ whenever $u\in H(z,z^{+})$.
	Because $u\in \Dshape(z,z^{+})=\udisc\cap H(z,z^{+})$, this implies that $A$ lies on the boundary arc of $\Dshape(z,z^{+})$ from $z^{+}$ to $z$.
	
	Consider the triangle $\triangle Auz$ and set
	\[
	\phi_1:=\Theta(A;u,z),\qquad \phi_2:=\Theta(z;A,u).
	\]
	Since $A,z,z^{+}\in\T$ and $A,u,z^{+}$ are collinear, the inscribed angle theorem gives
	\[
	\phi_1=\Theta(A;z,z^{+})=\frac{\alpha}{2}.
	\]
	Moreover, since $u$ lies between $A$ and $z^{+}$ on their common line, the rays from $u$ to $A$ and to $z^{+}$ are opposite, so
	\[
	\Theta(u;z,z^{+})=\pi-\Theta(u;A,z).
	\]
	Using that the angles in $\triangle Auz$ sum to $\pi$, we obtain
	\[
	\phi_2=\pi-\phi_1-\Theta(u;A,z)=\Theta(u;z,z^{+})-\frac{\alpha}{2}.
	\]
	Thus it suffices to bound $\phi_2$ from below.
	
	By the law of sines in $\triangle Auz$,
	\[
	\frac{\sin\phi_2}{\abs{Au}}=\frac{\sin\phi_1}{\abs{uz}}.
	\]
	Since $\phi_2\in(0,\pi)$ we have $\phi_2\ge \sin\phi_2$.
	Also, $\abs{A}=1$ and $\abs{u}\le 1-\varepsilon$ imply by the reverse triangle inequality that
	\[
	\abs{Au}\ge \bigl|\abs{A}-\abs{u}\bigr|\ge \varepsilon.
	\]
	Therefore,
	\begin{equation}\label{eq:phi2-basic}
		\phi_2 \ge \sin\phi_2 = \abs{Au}\frac{\sin\phi_1}{\abs{uz}}
		\ge \varepsilon\,\frac{\sin(\alpha/2)}{\abs{uz}}.
	\end{equation}
	
	If $\alpha\le\pi$, then trivially $\abs{uz}\le 2$, so \eqref{eq:phi2-basic} yields
	\[
	\phi_2\ge \frac{\varepsilon}{2}\sin\!\Big(\frac{\alpha}{2}\Big).
	\]
	If $\alpha>\pi$, then Lemma~\ref{lem:D-diameter} gives $\abs{uz}\le \abs{z^{+}-z}=2\sin(\alpha/2)$, hence
	\[
	\phi_2\ge \varepsilon\,\frac{\sin(\alpha/2)}{2\sin(\alpha/2)}=\frac{\varepsilon}{2}.
	\]
	This proves both cases.
\end{proof}

\begin{cor}[A uniform lower bound for subtended angles]\label{cor:angle-lowerbound}
	Let $z=e^{i\theta}$ and $z^{+}=e^{i\theta^{+}}$ be distinct points on $\T$ with $\alpha:=\theta^{+}-\theta\in(0,2\pi)$.
	Then for every $u\in \Dshape(z,z^{+})$,
	\[
	\Theta(u;z,z^{+})\ge \frac{\alpha}{2},
	\]
	with equality (for $u\neq z,z^{+}$) precisely when $u$ lies on the boundary arc of $\Dshape(z,z^{+})$ from $z^{+}$ to $z$.
	At the endpoints $u\in\{z,z^{+}\}$, equality holds by Definition~\ref{def:endpoint}.
\end{cor}

\begin{proof}
	If $u\in\{z,z^{+}\}$, the claim follows from Definition~\ref{def:endpoint}.
	
	Assume $u\notin\{z,z^{+}\}$.
	If $u$ lies on the boundary arc of $\Dshape(z,z^{+})$ from $z^{+}$ to $z$, then $u\in\T$ and $u$ is not on the open arc from $z$ to $z^{+}$ (counterclockwise), so Lemma~\ref{lem:half-gap} gives $\Theta(u;z,z^{+})=\alpha/2$.
	
	If $u$ lies on the chord $[z,z^{+}]$, then $u,z,z^{+}$ are collinear and $\Theta(u;z,z^{+})=\pi\ge \alpha/2$.
	
	Finally, if $u$ lies in the interior of $\Dshape(z,z^{+})$, repeat the construction in the proof of Lemma~\ref{lem:angle-gain}:
	let the line through $u$ and $z^{+}$ meet $\T$ again at $A\neq z^{+}$.
	Then, with $\phi_1,\phi_2$ as there, we have $\phi_1=\alpha/2$ and
	\[
	\phi_2=\Theta(u;z,z^{+})-\frac{\alpha}{2},
	\]
	where $\phi_2$ is an interior angle of the (nondegenerate) triangle $\triangle Auz$, hence $\phi_2>0$.
	Therefore $\Theta(u;z,z^{+})>\alpha/2$ for interior points.
\end{proof}

\section{Proof of Theorem~\ref{thm:incomplete_gap}}\label{sec:Proof of Theorem incomplete_gap}

\begin{proof}[Proof of Theorem~\ref{thm:incomplete_gap}]
	If $M=1$, then $\mathcal{L}(u)=(u-\zeta_1)^N$ and hence, for every $1\le j\le N$,
	\[
	\mathcal{L}_j(u)=\frac{\mathcal{L}(u)}{u-z_j}=(u-\zeta_1)^{N-1}.
	\]
	Therefore $\mathcal{L}_\lambda(u)=(u-\zeta_1)^{N-1}$, so all zeros of $\mathcal{L}_\lambda$ lie on $\T$ and $\mathcal{N}_\varepsilon=0$ for every $\varepsilon\in(0,1)$.
	The desired bound holds trivially. Hence we may assume $M\ge 2$.
	
	Let $z=e^{i\theta}$ and $z^{+}=e^{i\theta^{+}}$ be a consecutive \emph{distinct} pair of zeros of $\mathcal{L}$ realizing the maximal gap
	\[
	\alpha:=\theta^{+}-\theta = G.
	\]
	Let $w_1,\dots,w_{N-1}$ be the zeros of $\mathcal{L}_\lambda(u)$, counted with multiplicity.
	
	By Corollary~\ref{cor:q-in-D}, every zero $w_k$ of $\mathcal{L}_\lambda$ lies in $\Dshape(z,z^{+})$.
	In particular, by Corollary~\ref{cor:angle-lowerbound}, for each such $w_k$,
	\begin{equation}\label{eq:angle-lower}
		\Theta(w_k;z,z^{+})\ge \frac{\alpha}{2},
	\end{equation}
	with equality permitted at the endpoints via Definition~\ref{def:endpoint}.
	
	By Theorem~\ref{thm:incomplete_dual},
	\begin{equation}\label{eq:sum-angle}
		\sum_{k=1}^{N-1}\Theta(w_k;z,z^{+})=\pi+(N-2)\frac{\alpha}{2}.
	\end{equation}
	
	Let $\mathcal{N}_\varepsilon$ be as in \eqref{eq:def-Neps}.
	Split the sum \eqref{eq:sum-angle} into the $\mathcal{N}_\varepsilon$ terms with $\abs{w_k}<1-\varepsilon$ and the remaining terms.
	Using \eqref{eq:angle-lower} for the remaining $N-1-\mathcal{N}_\varepsilon$ terms and Lemma~\ref{lem:angle-gain} for the $\mathcal{N}_\varepsilon$ interior terms yields
	\[
	\sum_{k=1}^{N-1}\Theta(w_k;z,z^{+})
	\ge
	\mathcal{N}_\varepsilon\Big(\frac{\alpha}{2}+\delta\Big) + (N-1-\mathcal{N}_\varepsilon)\frac{\alpha}{2},
	\]
	where
	\[
	\delta=
	\begin{cases}
		\displaystyle \frac{\varepsilon}{2}\,\sin\!\Big(\frac{\alpha}{2}\Big), & \alpha\le \pi,\\[8pt]
		\displaystyle \frac{\varepsilon}{2}, & \alpha>\pi.
	\end{cases}
	\]
	Comparing with \eqref{eq:sum-angle} gives
	\[
	\mathcal{N}_\varepsilon\,\delta \le \pi-\frac{\alpha}{2}.
	\]
	
	\smallskip\noindent
	\textbf{Case 1: $\alpha\le \pi$.}
	Let $x:=\alpha/2\in(0,\pi/2]$. Then
	\[
	\mathcal{N}_\varepsilon \le \frac{\pi-x}{(\varepsilon/2)\sin x}
	=\frac{2}{\varepsilon}\cdot \frac{\pi-x}{\sin x}.
	\]
	Using the elementary bound $\sin x \ge x-\frac{x^3}{6}$ for $x\ge 0$, and the fact $x\le \pi/2<6/\pi$, we obtain
	\[
	\pi\sin x \ge \pi x-\frac{\pi}{6}x^3 \ge \pi x-x^2=x(\pi-x),
	\]
	so $\sin x \ge x(\pi-x)/\pi$, hence $(\pi-x)/\sin x \le \pi/x$. Consequently,
	\[
	\mathcal{N}_\varepsilon \le \frac{2}{\varepsilon}\cdot \frac{\pi}{x}
	=\frac{4\pi}{\varepsilon\,\alpha}.
	\]
	
	\smallskip\noindent
	\textbf{Case 2: $\alpha>\pi$.}
	Then $\delta=\varepsilon/2$ and
	\[
	\mathcal{N}_\varepsilon \le \frac{\pi-\alpha/2}{\varepsilon/2}=\frac{2\pi-\alpha}{\varepsilon}.
	\]
	For $\alpha\in[\pi,2\pi]$ we have $\alpha(2\pi-\alpha)\le \pi^2<4\pi$, hence $2\pi-\alpha\le 4\pi/\alpha$, and therefore
	\[
	\mathcal{N}_\varepsilon \le \frac{4\pi}{\varepsilon\,\alpha}.
	\]
	
	In both cases, since $\alpha=G$, we obtain $\mathcal{N}_\varepsilon \le 4\pi/(\varepsilon G)$, completing the proof.
\end{proof}

	\section*{Declaration of competing interest}
	The author declares no competing interests.
	
	\section*{Data availability}
	No data was used for the research described in the article.
	
	\section*{Acknowledgments}
Teng Zhang is supported by the China Scholarship Council, the Young Elite Scientists Sponsorship Program for PhD Students (China Association for Science and Technology), and the Fundamental Research Funds for the Central Universities at Xi'an Jiaotong University (Grant No.~xzy022024045).

\end{document}